\documentclass{amsart}

\usepackage{amssymb}
\usepackage{amsmath}
\usepackage[latin1]{inputenc}
\usepackage{eepic}
\usepackage{epsfig}
\usepackage{graphicx}
\usepackage[all]{xy}
\usepackage{enumerate}

\newtheorem{theo}{Theorem}
\newtheorem*{mainthm}{Main Theorem}
\newtheorem{lemm}[theo]{Lemma}
\newtheorem{coro}[theo]{Corollary}

\newtheorem{prop}[theo]{Proposition}
\newtheorem{quest}[]{Question}
\newtheorem{conj}[theo]{Conjecture}

\theoremstyle{remark}
\newtheorem{rema}[theo]{Remark}
\newtheorem{exam}[theo]{Example}

\newcommand{\KZ}{{K}_0({\rm Var}_k)}

\newcommand{\C}{\ensuremath{\mathbb{C}}}
\newcommand{\Z}{\ensuremath{\mathbb{Z}}}

\newcommand{\cpo}{\mathbb{P}^1_\C}
\newcommand{\cpd}{\mathbb{P}^2_\C}
\newcommand{\cpt}{\mathbb{P}^3_\C}
\newcommand{\Exc}{\text{Exc}}

\usepackage[colorlinks, linktocpage, pdfauthor={J. Sebag \& S. Lamy}, citecolor = blue, linkcolor = blue]{hyperref} 
\usepackage[all]{hypcap}

\begin{document}

\title{Birational self-maps and piecewise algebraic geometry}
\author{St\'ephane Lamy}
\author{Julien Sebag}
\address{St\'ephane Lamy, Institut de Math\'ematiques de Toulouse, Universit\'e Paul Sabatier, 118 route de Narbonne, 31062 Toulouse Cedex 9, France}
\email{slamy@math.univ-toulouse.fr}
\address{Julien Sebag, Irmar--Umr 6625 du CNRS, 263 Avenue du General Leclerc, CS 74205, 35042 Rennes Cedex, France}
\email{Julien.Sebag@univ-rennes1.fr} 

\begin{abstract} 
Let $X$ be a smooth projective complex variety, of dimension 3, whose Hodge numbers $h^{3,0}(X), h^{1,0}(X)$ both vanish. Let $f\colon X\dasharrow X$ be a birational map that induces an isomorphism  $U\cong V$ on (dense) open subvarieties $U,V$ of $X$. Then we show that the complex varieties $(X\setminus U)_{\mathrm{red}},(X\setminus V)_{\mathrm{red}}$ are piecewise isomorphic.
\end{abstract} 

\maketitle

\section{Introduction}
\label{intro}

In \cite{Gromov}, \textsc{Misha Gromov} asks the following question:

\begin{quest}
\label{Gromov}
Let $X,Y$ be algebraic varieties that admit embeddings to a third one, say $X\hookrightarrow Z$ and $Y\hookrightarrow Z$, such that the complements $Z\setminus X$ and $Z\setminus Y$ are biregular isomorphic. How far are $X$ and $Y$ from being birationally equivalent?\\
\end{quest}

Let $k$ be an arbitrary field. Let $\Z[{\rm Var}_k]$ be the free abelian group generated by the 
isomorphism classes of $k$-schemes of finite type. One defines the \emph{Grothendieck ring of varieties over $k$}, denoted by $\KZ$, as the quotient of 
$\Z[\mathrm{Var}_k]$ by the following ``scissor'' relations: whenever $X$ 
is a $k$-scheme of finite type, and $F$ is a closed subscheme of $X$, 
we impose that 
$$
[X]=[F]+[X\setminus F].
$$
The multiplication in $\KZ$ is defined by 
$$[X]\cdot[X']:=[X\times_k X'].$$
for any pair of $k$-schemes of finite type $X,X'$. 

Independently of \cite{Gromov}\footnote{The existence of a link between Question \ref{Gromov} and the structure of $\KZ$ is already mentioned in \cite[3G$'''$]{Gromov}.}, \textsc{Michael Larsen} and \textsc{Valery Lunts} ask, in \cite[Question 1.2]{Larsen-Lunts}, the following question in order to understand the ``geometric nature'' of  the Grothendieck ring of varieties:

\begin{quest}
\label{Larsen-Lunts}
Let $X,Y$ be reduced separated schemes of finite type over $k = \C$ such that $[X]=[Y]$ in $\KZ$. Is it true that $X,Y$ are piecewise isomorphic?

\end{quest}

Recall that two $k$-schemes of finite type $X,Y$ are called \emph{piecewise isomorphic}, if one can find a finite set $I$, a partition $(V_i)_{i\in I}$ of $X$ in (locally closed) subschemes of $X$, and  a partition $(W_i)_{i\in I}$ of $Y$ in (locally closed) subschemes of $Y$, such that, for any $i$, $i\in I$, there exists an isomorphism of $k$-schemes $f_i\colon (V_i)_{\mathrm{red}}\rightarrow (W_i)_{\mathrm{red}}$. Note that this property only depends on the reduced scheme structures of $X$ and $Y$. Up to now, this question is open in general (see \cite{Liu-Sebag}, \cite{Sebag-PAMS} for positive elements of answer in this direction, or \cite[Theorem 13.1]{HK} that uses a different point of view).\\

The fact that, without any condition on $k$, a positive answer to Question \ref{Larsen-Lunts} implies a precise answer to Question \ref{Gromov} is easy to see. We adopt the notations of Question \ref{Gromov}. Because of the relations lying in $\KZ$, we deduce that
$[Z\setminus X]=[Z\setminus Y]$, and that $[X]=[Y]$ in $\KZ$. If Question \ref{Larsen-Lunts} has a positive answer, we conclude that $X,Y$ are piecewise isomorphic. Because of the definition of a piecewise isomorphism, this remark proves in particular the existence of a bijection $\sigma$ between the set of the irreducible components of maximal dimension of respectively $X$ and $Y$, such that $(X_i)_{\mathrm{red}}$ and $(\sigma(X_i))_{\mathrm{red}}$ are birationally equivalent, for any irreducible component $X_i$ of $X$.\\

Besides, note that Question \ref{Larsen-Lunts} holds true in particular for any pair $X,Y$ of separated $k$-schemes of finite type of dimension at most 1, if $k$ is supposed to be algebraically closed of characteristic 0 (see \cite[Proposition 6]{Liu-Sebag}). \\

In the present article, we investigate the following conjectural reformulation of Question \ref{Gromov} based on Question \ref{Larsen-Lunts} (see also \cite[Remark 13]{Liu-Sebag}, or \cite{Sebag-PAMS}).

\begin{conj}
\label{la conjecture}
Let $k$ be an algebraically closed field of characteristic zero. Let $X$ be a separated scheme of finite type over $k$. Let $f\colon X\dasharrow X$ be a birational map that induces an isomorphism of $k$-schemes $U\cong V$ on dense open subsets $U,V$ of $X$. Then the $k$-schemes $(X\setminus U)_{\mathrm{red}},(X\setminus V)_{\mathrm{red}}$ are piecewise isomorphic.
\end{conj}

Note that, by a classical argument of birational geometry, if $X$ is a smooth projective scheme over $k$, of arbitrary dimension $d$, having no rational curves, then such a map $f$ is, in fact, an isomorphism of $k$-schemes. Thus Conjecture \ref{la conjecture} holds true in that case, but remains open, and particularly relevant, for separated $k$-schemes of finite type having an infinite number of rational curves. (The case where the separated $k$-scheme of finite type $X$ has only finitely many rational curves can be solved positively for example by \cite[Theorem 5]{Liu-Sebag}.)\\

Precisely, we prove the following theorem, which is the main result of this article (see Theorem \ref{le theoreme}, and Remark \ref{Gen thm} for a slightly more general statement).

\begin{mainthm}
Let $X$ be a smooth irreducible projective $\mathbb{C}$-variety, of dimension 3, whose Hodge numbers $h^{3,0}(X), h^{1,0}(X)$ both vanish. Let $f\colon X\dasharrow X$ be a birational map that induces a $\mathbb{C}$-isomorphism $U\cong V$ on dense open subvarieties $U,V$ of $X$. Then the $\mathbb{C}$-varieties $(X\setminus U)_{\mathrm{red}}$, $(X\setminus V)_{\mathrm{red}}$ are piecewise isomorphic.
\end{mainthm}

The proof of such a statement is based on the following ingredients.

\begin{enumerate}[$a$)]
\item The use of the Weak Factorization Theorem for birational maps (see \cite{Wlo}) and its combinatoric;
\item The use of intermediate Jacobians to control the weak factorizations of $f$ (see \cite{Clemens-Griffiths});
\item The positive answer to Larsen--Lunts' question in the case of curves (see \cite[Proposition 6]{Liu-Sebag}).
\end{enumerate}

This article is mostly self-contained, for the convenience of the reader. 
Sections \ref{PAG} and \ref{JI} give an overview of some tools of birational, piecewise algebraic and complex geometry that we use in the proof of Theorem \ref{le theoreme}. Besides, we establish in these sections some technical lemmas, sometimes not explicitly described in the literature, that we need, in a crucial way, in the proof of Theorem \ref{le theoreme} (e.g., see Lemma \ref{dimension}, Corollary \ref{piece weak}, Propositions \ref{cas des courbes}, \ref{jacobienne}, Corollary \ref{corollaire jacobienne2}). 
Then Section \ref{MAIN} is devoted to our main theorem and its proof (see Theorem \ref{le theoreme}).
In Sections \ref{EXPLE} and \ref{COM} we give respectively some explicit examples and some further comments. \\

The authors would thank Qing Liu and Fran\c{c}ois Loeser for their comments on a first version of this article.

\section{Conventions}

 Let $k$ be a field. A \emph{curve} (resp. \emph{surface}) is a separated $k$-scheme of finite type purely of dimension 1 (resp. 2). More generally, a $k$-\emph{variety} is a separated $k$-scheme of finite type. If $x$ is a point in a $k$-variety, we define its dimension as the transcendence degree over $k$ of the residue field $\kappa(x)$ of $X$ at $x$. \\

Two reduced $k$-varieties $X,Y$ are \emph{birationally equivalent} if there exist a dense open subscheme of $X$ and a dense open subscheme of $Y$ that are isomorphic as $k$-schemes. We say that two reduced $k$-varieties $X,Y$ are  \emph{stably birational} if there exist integers $m,n\in\mathbf{N}$ such that $X\times_k \mathbb{P}^m_k, Y\times_k\mathbb{P}^n_k$ are birationally equivalent. If $k$ is algebraically closed, we say that an integral $k$-variety $X$ is \emph{rational}, if it is birationally equivalent to $\mathbb{P}^{\mathrm{dim}(X)}_k$. More generally, if $k$ is algebraically closed, a reduced $k$-variety is said rational if its irreducible components are all rational. 
An integral $k$-variety $X$, of dimension $d$, is said \emph{uniruled} if there exists a $k$-variety $Y$ of dimension $d-1$, and a dominant rational map $Y\times_k \mathbb{P}^1_k\dasharrow X$. If $k$ 
is an algebraically closed field of characteristic zero, and if $d=2$, a smooth projective $k$-variety $X$ is uniruled if and only if it is ruled. \\

If $X,Y$ are two piecewise isomorphic $k$-varieties, the datum of the family of the $k$-isomorphisms $(f_i)_{i\in I}$ (given, by definition, between the reduced locally closed strata of $X,Y$) is called a \emph{piecewise isomorphism}. \\

If $k$ is an algebraically closed field, we say that a $k$-variety is an abelian variety if $X$ is a smooth irreducible projective abelian group scheme over $k$.\\

If $X$ is a $\mathbb{C}$-variety, we denote by $X^{\mathrm{an}}:=X(\mathbb{C})$ the \emph{analytic space associated to} $X$.\\

In $\KZ$, we denote by $\mathbb{L}$ the class of the affine line, and by $[X]$ the class of the $k$-scheme of finite type $X$. Let $A$ be a ring. A ring morphism $\KZ \rightarrow A$ is called a \emph{motivic measure}.

\section{Birational and piecewise algebraic geometry}
\label{PAG}

In this section, we review some tools of birational and piecewise algebraic geometry that we use in the proof of Theorem \ref{le theoreme}. We also establish some useful preliminary lemmas.

\subsection{Birational maps and the Weak Factorization Theorem}

In this paragraph, we review some well-known results on birational maps such as the Weak Factorization Theorem established in \cite[Theorem 0.1.1]{Wlo} and we assume that $k$ is an algebraically closed field of characteristic zero. 
Let $X,Y$ be two $k$-varieties. 
A rational map $f\colon X\dasharrow Y$ that induces an isomorphism of $k$-schemes $U\cong V$ from a dense open subvariety of $X$ onto a dense open subvariety of $Y$ is called a \emph{birational map}. 
In particular, there exists a unique maximal open subvariety $\mathrm{Dom}(f)$ of $X$, such that $f_{\vert_{\mathrm{Dom}(f)}}\colon \mathrm{Dom}(f)\rightarrow Y$ is a morphism of $k$-schemes. 
If $x\in X$, we say that a birational map $f\colon X\dasharrow Y$ is a \emph{local isomorphism} (for the Zariski topology) at $x$ if there exists an open subvariety $U$ of $\mathrm{Dom}(f)$ that contains $x$ such that the restriction of $f_{\vert_{\mathrm{Dom}(f)}}$ to $U$ is an open immersion of $k$-schemes. Note that, if $f\colon X\dasharrow Y$ is a local isomorphism at $x$, then $x\in\mathrm{Dom}(f)$. Remark also that the set of points of $X$, at which a birational map $f\colon X\dasharrow Y$ is \emph{not} a local isomorphism, is a closed subset of $X$, that we call \emph{exceptional set} of $f$. We denote it by $\mathrm{Exc}(f)$.\\

The following lemmas are recalled for convenience for the reader. Lemma \ref{unireglage} should be kept in mind for the proof of Theorem \ref{le theoreme}.

\begin{lemm}
\label{equivalence top}
Let $X,Y$ be two integral $k$-varieties. Let $f\colon X\dasharrow Y$ be a birational map. Let $x\in \mathrm{Dom}(f)$ be a point. Let us set $y:=f(x)$. The following assertions are equivalent.
\begin{enumerate}[a\emph{)}]
\item The map $f$ is a local isomorphism at $x$;
\item The morphism $f_{\vert_{\mathrm{Dom}(f)}}$ induces an isomorphism of $k$-algebras $\mathcal{O}_{Y,y}\rightarrow \mathcal{O}_{X,x}$;
\item The morphism $f_{\vert_{\mathrm{Dom}(f)}}$ is \'etale at $x$.
\item The morphism $f_{\vert_{\mathrm{Dom}(f)}}$ induces an isomorphism of $k$-algebras $\widehat{\mathcal{O}_{Y,y}}\rightarrow \widehat{\mathcal{O}_{X,x}}$.
\end{enumerate}
\end{lemm}

\begin{proof}
Since $f_{\vert_{\mathrm{Dom}(f)}}$ is dominant, the induced morphism of local $k$-algebras $f_x\colon \mathcal{O}_{Y,y}\rightarrow \mathcal{O}_{X,x}$
is injective. 

Assume that $f_{\vert_{\mathrm{Dom}(f)}}$ is \'etale at $x$. Thus, the local morphism $f_x$ is flat, hence faithfully flat. Since $f$ is birational, we deduce that $f_x$ is an isomorphism. It shows that  point $c$) implies point $b$). Besides, $f_{\vert_{\mathrm{Dom}(f)}}$ then induces an isomorphism of $k$-algebras $\kappa(y)\rightarrow \kappa(x)$. So point $c$) is equivalent to point $d$).

\if false
Hence, the induced morphism
there  morphism $f_{\vert_{\mathrm{Dom}(f)}}$ is quasi-finite at $x$. Then there exist an open subvariety $U$ of $X$, that contains $x$ and contained in $\mathrm{Dom}(f)$, and an open subvariety $V$ of $Y$, which contains $f(U)$, such that $f_{\vert_{U}}\colon U \rightarrow V$ is quasi-finite, and birational by assumption. By Grothendieck's form of Zariski's Main Theorem applied to the morphism $f_{\vert_{U}}$, we conclude that $f_{\vert_{U}}\colon U\rightarrow Y$ is an open immersion, or equivalently that $f\colon X\dasharrow Y$ is a local isomorphism at $x$.
\fi
\end{proof}

\begin{rema}
Let $n\in \mathbb{N}$, $n\geq 1$. To emphasize Lemma \ref{equivalence top}, consider the birational $\mathbb{C}$-morphism $f\colon \mathbb{A}_\mathbb{C}^n \rightarrow \mathbb{A}_\mathbb{C}^n$ defined by the datum of $(f_i)_{1\leq i\leq n}\in (\mathbb{C}[x_1,\ldots,x_n])^n$. Then we deduce from the Formal Inverse Function Theorem that the set of the $\mathbb{C}$-points of $\mathbb{A}_\mathbb{C}^n$ at which $f$ is not a local isomorphism is defined as the $\mathbb{C}$-points of the hypersurface of $\mathbb{A}_\mathbb{C}^n$ defined by the determinant of the Jacobian matrix of $f$, i.e., defined by the equation
$$
\mathrm{det}((\partial f_i/\partial x_j)_{1\leq i,j\leq n})=0.
$$
\end{rema}

\begin{lemm}
\label{unireglage}
Let $f\colon X\dasharrow Y$ be a birational map between two integral $k$-varieties. Assume that $Y$ is normal. Let $U_0\subset \mathrm{Dom}(f)$ be the open set of points of $X$ at which $f$ is a local isomorphism. Let us set $V_0 := f(U_0)$. Then
\begin{enumerate}[a\emph{)}]
\item The morphism of $k$-schemes $f_{\vert_{U_0}}\colon U_0\rightarrow Y$ is an open immersion. So the map $f$ induces an isomorphism of $k$-schemes $U_0 \rightarrow V_0$.

\item Assume that $X,Y$ are proper and smooth over $k$. Each irreducible component of codimension 1 in $(X \setminus U_0)_{\mathrm{red}}$ is ruled. 
\end{enumerate}
\end{lemm}

\begin{proof}
$a$) From the definition of $U_0$ and from Lemma \ref{equivalence top}, we deduce that the restriction of the morphism $f_{\vert_{\mathrm{Dom}(f)}}$ to $U_0$ is separated, birational and locally quasi-finite. From Grothendieck's form of Zariski's Main Theorem applied to the morphism $f_{\vert_{U_0}}$, we conclude that $f_{\vert_{U_0}}$ is an open immersion. So it induces an isomorphism of $k$-schemes from $U_0$ to $V_0$.

$b$) By eliminating the indeterminacies of $f$ (e.g., see \cite[Lemma 1.3.1]{Wlo}), we find a composite of blow-ups $\pi\colon Z\rightarrow X$, with smooth centers disjoint from $U_0$, and a birational morphism $g\colon Z\rightarrow Y$ such that $g=f\circ \pi$. Let $\eta$ be the generic point of an irreducible component $X_\eta$ of maximal dimension in $(X \setminus U_0)_{\mathrm{red}}$. Note that $\eta\in \mathrm{Dom}(f)$. The strict transform by $\pi$ of $X_\eta$ (considered with its reduced scheme structure) is a divisor $Z_\eta$ of $Z$, birationally equivalent to $X_\eta$. But, because of the choice of $U_0$ and the dimension of $\eta$, the subvariety $Z_\eta$ necessarily is an irreducible component of the exceptional locus of the morphism $g$. By Abhyankar's Theorem on exceptional sets of birational morphisms, we deduce that $Z_\eta$ is ruled, i.e., birationally equivalent to $\mathbb{P}^1_\mathbb{C}\times_\mathbb{C} S$ where $S$ is an integral $k$-variety of dimension $
 \mathrm{dim}(X)-2$. So $X
 _\eta$ is ruled.
\end{proof}

\begin{rema}
Note that Lemma \ref{unireglage} in particular shows that the set of points of $X$ at which $f$ is a local isomorphism is the largest open subset $U$ of $X$ such that $f_{\vert_U}$ is an isomorphism of $k$-schemes.

\end{rema}

The Weak Factorization Theorem was a great step in the understanding of the geometry of birational maps. 
We recall below the statement of this theorem.\\

Let $X,Y$ be smooth irreducible proper (resp. projective) $k$-varieties. Let $f\colon X\dasharrow Y$ be a birational map that induces an isomorphism of $k$-schemes $U\rightarrow V$ on a dense open subvariety $U$ of $X$ onto a dense open subvariety $V$ of $Y$. We say that $f$ admits a \emph{weak factorization} if there exist birational morphisms $b_i\colon Z_i\rightarrow X_i$, $c_i\colon Z_i\rightarrow X_{i+1}$, for $i$, $1\leq i\leq n$ between smooth irreducible proper (resp. projective) $k$-varieties such that:

$$
\xymatrix{&Z_1\ar[dl]_{b_1}\ar[dr]^{c_1}&&&&Z_n\ar[dl]_{b_n}\ar[dr]^{c_n}&\\
X_1&&X_2&\ldots&X_n&&X_{n+1}}
$$

\begin{enumerate}[$a$)]
\item $X:=X_1$, $Y:=X_{n+1}$;
\item For any $i$, $1\leq i\leq n$, the $b_i$ and the $c_i$ are composites of blow-ups, with smooth irreducible centers, disjoint from $U$ (note that, in our notation and for convenience in the following computations, $c_n\colon Z_n\rightarrow Y$ or $b_1\colon Z_1\rightarrow X$ could be the blow-up with empty center, i.e., $c_n=\mathrm{id}_Y$ and $b_1=\mathrm{id}_X$).
\end{enumerate}

\begin{theo}[{\cite[Theorem 0.1.1]{Wlo}}]
\label{factorization faible}
Let $X,Y$ be two smooth irreducible proper $k$-varieties. Let $f\colon X\dasharrow Y$ be a birational map. Then $f$ admits a weak factorization.
\end{theo}

\begin{exam}
In dimension 2, a stronger version of Theorem \ref{factorization faible} goes back to \textsc{Oscar Zariski}. Let $X,Y$ be two  smooth irreducible projective $k$-surfaces, and let $f\colon X\dasharrow Y$ be a birational map. The elimination of the indeterminacies of $f$ produces a weak factorization of $f$. Indeed, in that case, one knows that there exist composites of blow-ups of points $b\colon Z\rightarrow X, c\colon Z\rightarrow  Y$ making the following diagram commute (see \cite[Theorem \textbf{9}.2.7]{Liu}):
$$
\xymatrix{&Z\ar[dl]_{b}\ar[dr]^{c}&\\
X\ar@{-->}[rr]_{{f}}&&Y}
$$
The existence of such a \emph{strong factorization} in any dimension is still a conjecture.
\end{exam}

The following remark should be kept in mind for the proof of Theorem \ref{le theoreme}.

\begin{rema} \label{rem:uniruled}
Let $b_{ij}$ be a blow-up that appears in the expression of $b_i$. Then two cases occur.
\begin{enumerate}[$a$)]
\item First, the strict transform of its exceptional divisor might be equal for some $j'$ to the exceptional divisor of a blow-up $c_{i' j'}$ appearing in the expression of $c_{i'}$. Note that, in that case, the strict transform is then contracted by $c_{i' j'}$.
\item On the contrary, it can happen that no blow-up appearing in any $c_i$ contracts the strict transform of its exceptional divisor. In that case, we can consider that strict transform in $Y$. Then it is equal to a divisor of $Y$ contained in $Y\setminus V$. 
Note that if $V$ is chosen as the maximal open subset on which $f^{-1}$ is well-defined and injective, then all the irreducible components, of dimension $\mathrm{dim}(Y)-1$, in $Y\setminus V$ are realized this way, and in particular are ruled. (This last assertion follows from Lemma \ref{unireglage}, or from the Weak Factorization Theorem).
\end{enumerate}
\end{rema}


\subsection{Cohomology theories and examples of motivic measures}
\label{realisation de Hodge-Deligne}

In this paragraph, we assume that $k$ is a subfield of the field of complex numbers, except when stated otherwise. Let $X$ be a smooth proper $k$-variety. We denote by 
$H^\ast_{\mathrm{dR}}(X)$ the \emph{algebraic de Rham cohomology} of $X$. Then 
$$\dim_k H^n_{\mathrm{dR}}(X)=\sum_{i+j=n}\dim_k H^i(X,\Omega^j_{X/k})$$
(Hodge's decomposition theorem) for any integer $n$. Let us denote by $H^{p,q}(X)$ the $\mathbb{C}$-vector space $H^{q}(X,\Omega^p_{X/k})\otimes_k \mathbb{C}$, for any pair $p,q$, $0\leq p,q\leq \mathrm{dim}(X)$. The integer $$h^{p,q}(X):=\dim_\mathbb{C} H^{p,q}(X)$$ is called the 
$(p,q)$-\emph{Hodge number} of $X$. We
denote by $H^\ast(X):=H^\ast(X^{\mathrm{an}},\mathbb{Z})$ the \emph{integral singular
cohomology} of $X^{\mathrm{an}}$. By Grothendieck's results 
(\cite{Groth_dR}, 
Theorem $1'$), we know that the canonical map $H^\ast_{\mathrm{dR}}(X)\otimes_k \mathbb{C}\simeq H^\ast(X)\otimes_\mathbb{Z}\mathbb{C}$ is an isomorphism. For $i$, $1\leq i\leq 2\mathrm{dim}(X)$, we denote by 
$$
b_i(X):=\mathrm{dim}_\mathbb{Q} H^i(X)\otimes_\mathbb{Z}\mathbb{Q}=\mathrm{dim}_\mathbb{C} H^i(X)\otimes_\mathbb{Z}\mathbb{C}
$$
the $i$-th \emph{Betti number} of $X$.\\

By Deligne's
mixed Hodge theory, the assignment to a $\mathbb{C}$-variety $X$ of its \emph{Hodge polynomial}
$X\mapsto H_X(u,v)$ defines a unique ring morphism
$$
K_0(Var_{\mathbb{C}})\rightarrow \mathbb{Z}[u,v]
$$
(see \cite[Theorem 1.1]{Srinivas} or \cite[\S 6.1]{Denef-Loeser}). Recall that, if $X$ is a smooth proper $\mathbb{C}$-variety, its Hodge polynomial is defined as
$$
HD_X(u,v):=\sum_{p,q\geq 0} h^{p,q}(X)u^pv^q.
$$ 
We call 
\emph{Poincar\'e polynomial} of a $\mathbb{C}$-variety $X$ the polynomial $P_X\in 
\mathbb{Z}[u]$, defined by $P_X(u):=HD_X(u,u)$. This assignment defines
a unique ring morphism
$$
K_0(Var_{\mathbb{C}})\rightarrow \mathbb{Z}[u].
$$
Recall that $P_X(u)=\sum_{n=0}^{2\dim X}\dim_\mathbb{C} H^n_{\mathrm{dR}}(X)u^n=\sum_{n=0}^{2\dim X} b_n(X)u^n$
if $X$ is a smooth proper $\mathbb{C}$-variety, by Hodge Decomposition Theorem. \\

One defines the \emph{Euler characteristic} with compact support to be the ring morphism
$$
\chi_c\colon K_0(Var_{\mathbb{C}})\rightarrow \mathbb{Z}
$$
defined by $\chi_c([X])=P_X(-1)$, for any $\mathbb{C}$-variety $X$. If $X$ is a smooth proper $\mathbb{C}$-variety, then
$$
\chi_{c}([X]):= \sum_{n=0}^{2\dim X}(-1)^nb_n(X).
$$ 
 This definition corresponds to the usual topological Euler characteristic associated to $X^{\mathrm{an}}$.

\begin{exam}[{see \cite[Proposition 2.3(a)]{Ivorra-Sebag}}]
Let $X,Y$ be two connected $k$-curves. Then $X,Y$ are piecewise isomorphic if and only if $X_{\mathrm{red}},Y_{\mathrm{red}}$ are birationally equivalent with $\chi_c(X)=\chi_c(Y)$.
\end{exam}

These constructions have the following useful consequence. 

\begin{lemm}
\label{dimension}
Let $k$ be an arbitrary field. Let us fix an algebraic closure $\bar{k}$ of $k$. Let $X,Y$ be $k$-varieties such that $[X]=[Y]$ in $\KZ$. Then the dimension of $X$ and $Y$ are equal. Besides, the number of irreducible components of maximal dimension of $X_{\bar{k}}$ is equal to that of $Y_{\bar{k}}$. 
\end{lemm}

\begin{proof}
See, for example, \cite[ Proposition 4, Corollary 5]{Liu-Sebag}. Below, we give a proof when $k$ equals $\mathbb{C}$, that is the framework of Theorem \ref{le theoreme}.

By the scissor relations in $\KZ$ and by the assumption, we have $[X_{\mathrm{red}}]=[X]=[Y]=[Y_{\mathrm{red}}]$. Since the statement only depends on the reduced scheme structures of $X$ and $Y$, we can assume that $X,Y$ are reduced $\mathbb{C}$-varieties. Let $X_1,\ldots,X_m$ (resp. $Y_1,\ldots, Y_n$) be the irreducible components of $X$ (resp. $Y$). By Nagata's Compactification Theorem and Hironaka's Theorem on resolution of singularities, we can find smooth irreducible proper $k$-varieties $\bar{X}_1,\ldots,\bar{X}_n$ (resp. $\bar{Y}_1,\ldots,\bar{Y}_n$) such that, for any $i$, $1\leq i\leq m$ (resp. $j$, $1\leq j\leq n$), $\bar{X}_i$ is birationally equivalent to $X_i$ (resp. $\bar{Y}_j$ is birationally equivalent to $Y_j$), and smooth irreducible proper $k$-varieties $C_1,\ldots,C_s,D_1,\ldots, D_t$ that transform the equality of classes $[X]=[Y]$ in $\KZ$ in the following one:
$$
\sum_{i=1}^m[\bar{X_i}]+\sum_{\alpha=1}^s\varepsilon_\alpha[C_\alpha]=\sum_{j=1}^n[\bar{Y_j}]+
\sum_{\beta=1}^t\varepsilon_\beta[D_\beta],
$$
with $\varepsilon_\alpha,\varepsilon_\beta\in\{1,-1\}$. By computing the degrees and the leading coefficients of the Poincar\'e polynomials associated to the left and right-hand terms, we obtain the conclusion of the statement. Indeed, by Serre Duality, one
knows that:
$$
h^{\mathrm{dim}(Z),\mathrm{dim}(Z)}(Z)=b_{2\mathrm{dim}(Z)}(Z)=1,
$$
for any smooth irreducible proper $k$-variety $Z$.
\end{proof}

We can also prove the following lemma.

\begin{lemm}
\label{piece weak0}
Let $k$ be an algebraically closed field of characteristic zero. Let $X,Y$ be piecewise isomorphic smooth irreducible proper $k$-varieties. Let $f\colon X\dasharrow Y$ be a birational map. Let 
$$
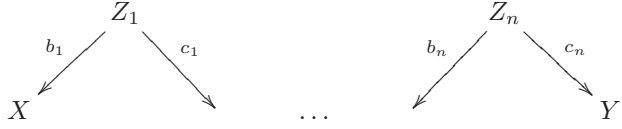
\xymatrix{&Z_1\ar[dl]_{b_1}\ar[dr]^{c_1}&&&&Z_n\ar[dl]_{b_n}\ar[dr]^{c_n}&\\
X&&&\ldots&&&Y}
$$
be a weak factorization of $f$. If $1\leq i\leq n$, let us denote by $C_{ij}$, for $1\leq j\leq s_i$, the (non-empty) smooth irreducible center (of codimension $\alpha_j$) of the blow-ups defining $b_i$, and by $D_{ij}$, for $1\leq j\leq t_i$, the (non-empty) smooth irreducible centers (of codimension $\beta_j$) of the blow-ups defining $c_i$.Then 
\begin{equation}
\label{piece weak1}
\mathbb{L}\sum_{i=1}^n\sum_{j=1}^{s_i}[C_{ij}][\mathbb{P}^{\alpha_j-2}_k]=\mathbb{L}\sum_{i=1}^n\sum_{j=1}^{t_i}[D_{ij}][\mathbb{P}^{\beta_j-2}_k].
\end{equation}

\end{lemm}

\begin{proof}
If $b_1=\mathrm{id}_X$, or $c_n=\mathrm{id}_Y$, eliminate $X=X_1$ or $Y=X_{n+1}$ from the diagram. Let $V$ be a smooth irreducible proper $k$-variety. Let $W$ be the blow-up of $V$, with smooth irreducible center $C$ of codimension $\alpha$ in $V$. Then, by the blow-up formula in $\KZ$, we deduce that:
$$
[W]-[C][\mathbb{P}_{k}^{\alpha -1}]=[V]-[C],
$$
or equivalently, using the identity $\mathbb{P}^\gamma_k = \sum_{i=0}^\gamma \mathbb{L}^i$,
$$
[W]-\mathbb{L}[C][\mathbb{P}_{k}^{\alpha -2}]=[V]
$$
(see, for example, \cite[Example 1]{Liu-Sebag}). Applying this formula to each blow-up appearing in the $b_i$, and to each blow-up
appearing in the $c_j$, we deduce the result by adding all the obtained relations.

\end{proof}

\begin{coro}
\label{piece weak}
Let $X,Y$ be piecewise isomorphic smooth irreducible proper $\mathbb{C}$-varieties. Let $f\colon X\dasharrow Y$ be a birational map. Let 
$$
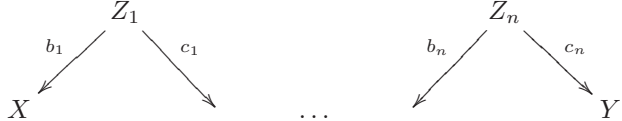
\xymatrix{&Z_1\ar[dl]_{b_1}\ar[dr]^{c_1}&&&&Z_n\ar[dl]_{b_n}\ar[dr]^{c_n}&\\
X&&&\ldots&&&Y}
$$
be a weak factorization of $f$. Then the number of all the (non-empty) blow-ups appearing in the expression of the $b_i$, $1\leq i\leq n$, is equal to the number of all (non-empty) the blow-ups appearing in the expression of the $c_i$, $1\leq i\leq n$.
\end{coro}

\begin{proof}
The proof follows from computing the leading coefficients of Poincar\'e polynomials in the formula of Lemma \ref{piece weak0}.
\end{proof}

\begin{rema}
\begin{enumerate}[$a$)]
\item With the notations of Lemma \ref{piece weak0}, assume that the dimension of $X$ (and so that of $Y$ by Lemma \ref{dimension}) is equal to 3. By looking at the proof of Lemma \ref{piece weak0}, we see that (\ref{piece weak1}) can be refined in:
\begin{equation}
\label{piece weak2}
\mathbb{L}\left(\sum_{i=1}^{n}\left(m_i(\mathbb{L}+1)+\sum_{\ell}[C_{ij_\ell}]\right)\right)=\mathbb{L}\left(\sum_{i=1}^{n}\left(n_i(\mathbb{L}+1)+\sum_{\ell'}[D_{ij_{\ell'}}]\right)\right)
\end{equation}
where $m_i$ (resp. $n_i$) is the number of blow-ups appearing in the expression of $b_i$ (resp. $c_i$), with centers a rational curve (i.e., $\mathbb{P}^1_k$) or a point, and where the $C_{ij_\ell}, D_{ij_{\ell'}}$ are smooth irreducible projective $k$-curves, whose genus $g$ is at least 1. If we compute Poincar\'e polynomials of the left and right-hand terms, then we deduce, by comparing the coefficients of the monomial of degree 3, that:
$$
\sum_{i=1}^{n}\sum_{\ell}g(C_{ij_\ell})=\sum_{i=1}^{n}\sum_{\ell'}g(D_{ij_{\ell'}}),
$$
where $g(\ast)$ is the genus of the $k$-curve $\ast$. In Proposition \ref{jacobienne}, we will see how to improve this result by introducing intermediate Jacobians.

\item One can prove Corollary \ref{piece weak} directly and independently from Lemma \ref{piece weak0} by computing Picard numbers, which increase by one under blow-up of an irreducible center.

\item Corollary \ref{piece weak} can be formulated when $k$ is any algebraically closed field of characteristic zero.
\end{enumerate}
\end{rema}

\subsection{Larsen--Lunts' motivic measure and some consequences}

In this paragraph we assume that $k$ is an algebraically closed field of characteristic zero. In \cite{Larsen-Lunts}, \textsc{M. Larsen} and \textsc{V. Lunts} have constructed a motivic measure that is a very important piece in the understanding of the Grothendieck ring of varieties and its structure. Precisely we have:

\begin{theo}[Larsen--Lunts]
Let $\mathbb{Z}[\mathfrak{sb}]$ be the ring generated by the classes of smooth irreducible proper $k$-varieties for the stably birational equivalence relation. 
There exists a unique surjective ring morphism
$$
\mathrm{SB}\colon \KZ\rightarrow \mathbb{Z}[\mathfrak{sb}],
$$
that assigns to the class in $\KZ$ of a smooth irreducible proper $k$-variety $X$ its stably birational class in $\mathbb{Z}[\mathfrak{sb}]$. Besides, its kernel is the ideal of $\KZ$ generated by $\mathbb{L}$.
\end{theo}

It is important to note that the description of the image of the morphism $\mathrm{SB}$ in terms of stably birational classes of $k$-varieties only holds for  classes of smooth irreducible proper $k$-varieties, which form a set of generators of the ring $\KZ$. The image of a $k$-variety which is not irreducible proper and smooth is defined as a linear combination of stably birational equivalence classes of smooth irreducible proper $k$-varieties. In particular, the image by $\mathrm{SB}$ of such a variety can be ``non-effective''. To illustrate  this point, consider the following example: It is easy to show that, if $Y$ is a rational $k$-variety of dimension 1, then \cite[Proposition 7]{Liu-Sebag} implies the following formula
$$
\mathrm{SB}([Y])=\chi_c([Y])-c_Y,
$$
where $c_Y$ is the number of irreducible components of dimension 1 of $Y$. See also \cite[Remark 1.6]{Ivorra-Sebag}. (Note that the canonical 
homomorphism $\Z\to \KZ$ is injective.) \\

The realization morphism has the following useful consequence.

\begin{prop}[{\cite[Proposition 6]{Liu-Sebag}}]
\label{cas des courbes}

Let $X,Y$ be $k$-varieties such that $[X]=[Y]$ in $\KZ$. Assume that the dimension of $X$ is at most 1. Then $X,Y$ are piecewise isomorphic.
\end{prop}

\begin{proof}
The proof below follows the arguments of the proof of \cite[Proposition 6]{Liu-Sebag}. Since the statement only depends on the reduced scheme structures of $X$ and $Y$, we can suppose that $X$ and $Y$ are reduced. Assume first that the dimension of $X$ is one. By Lemma \ref{dimension} the dimension of $Y$ is equal to 1 too. Let $\bar{X}_1,\ldots,\bar{X}_n$ (resp. $\bar{Y}_1,\ldots, \bar{Y}_n$) be smooth irreducible projective models of the irreducible components $X_1,\ldots,X_n$ (resp. $Y_1,\ldots, Y_n$) of $X$ (resp. of $Y$). So, the relation $[X]=[Y]$ can be transformed,  in $\KZ$, in the following one:
$$
\sum_{i=1}^n[\bar{X}_i]=\sum_{i=1}^n[\bar{Y}_i]+m,
$$
with $m\in\mathbb{Z}$. By applying the morphism $\mathrm{SB}$, we conclude that $m=0$, and, up to renumbering, that $\mathrm{SB}(\bar{X}_i)=\mathrm{SB}(\bar{Y}_i)$. Thus $\bar{X}_i,\bar{Y}_i$ are stably birational, for any $i$, $1\leq i\leq n$. Let us fix $i$, $1\leq i\leq n$. From a birational map $\bar{X}_i\times_k\mathbb{P}^s_k\dasharrow \bar{Y}_i\times_k \mathbb{P}^s_k$, we deduce, if $\bar{X}_i$ is rational (i.e., $k$-isomorphic to $\mathbb{P}^1_k$, because of the assumption on $k$), the existence of a dominant morphism of $k$-schemes $\mathbb{P}^1_k\rightarrow \bar{Y}_i$. Thus, by L\"uroth Problem, we conclude that $\bar{Y}_i$ is $k$-isomorphic to $\mathbb{P}^1_k$, and that $X_i,Y_i$ are birationally equivalent. If $X_i$ is not rational, the symmetry of the argument proves that $Y_i$ is not rational. By \cite[Theorem 2]{Liu-Sebag}, there exists an isomorphism of $k$-schemes $\bar{X}_i\rightarrow \bar{Y}_i$. So, $
X_i,Y_i$ are again birationally equivalent. These remarks imply the existence of dense open subvarieties $U$ of $X$ and $V$ of $Y$, and of a $k$-isomorphism $f_1\colon U\rightarrow V$. The complements of $U$ in $X$ and of $V$ in $Y$ are disjoint sums of points. A (non canonical) $k$-isomorphism $f_0\colon (X\setminus U)_{\mathrm{red}}\rightarrow (Y\setminus V)_{\mathrm{red}}$ exists if and only if the cardinals of the underlying sets of $(X\setminus U)_{\mathrm{red}}$ and $(Y\setminus V)_{\mathrm{red}}$ are equal. But this is the case thanks to the equality $\mathrm{Card}((X\setminus U)_{\mathrm{red}})=\chi_c([(X\setminus U)_{\mathrm{red}}])=\chi_c([(Y\setminus V)_{\mathrm{red}}])=\mathrm{Card}((X\setminus U)_{\mathrm{red}})$. 

If the dimension of $X$ is zero, then $X,Y$ are disjoint sums of points. We conclude in that case directly by computing Euler characteristics.
\end{proof}

\begin{rema}
Assume that $\mathbb{L}$ is not a zero divisor in $\KZ$. Then it follows from the formula of Lemma \ref{piece weak0} that
$$
\sum_{i=1}^n\sum_{j=1}^{s_i}[C_{ij}][\mathbb{P}^{\alpha_j-2}_k]=\sum_{i=1}^n\sum_{j=1}^{t_i}[D_{ij}][\mathbb{P}^{\beta_j-2}_k].
$$
If all the $C_{ij}, D_{ij}$ are smooth irreducible projective $k$-curves or points, it follows from Proposition \ref{cas des courbes} that there exists a bijection $\sigma$ from the set of the $C_{ij}$ to the set of the $D_{ij}$ such that, for any $C_{ij}$, the curves $C_{ij},\sigma(C_{ij})$ are $k$-isomorphic. See \S \ref{COM} for a discussion around cancellation problems.
\end{rema}

\section{Jacobians}
\label{JI}

\subsection{Principally polarized abelian varieties}
Recall that a \emph{polarization} on a complex torus $V / \Lambda$ is a positive definite hermitian form $\mathcal{H}$ on $V$ such that the imaginary part $\text{Im} \mathcal H$ takes integral values on $\Lambda \times \Lambda$. 
Equivalently a polarization is given by an ample line bundle $L$ on $X$; the first class Chern $c_1(L)$ then corresponds to $\text{Im} \mathcal H$.
A polarization is called \emph{principal} if the Hermitian form is unimodular, or equivalently if $H^0(X,L) = \C$. 
The pair $(X, \mathcal H)$, or equivalently $(X,L)$, is called a \emph{principally polarized abelian variety}. In this situation a divisor associated with a section of $L$ is called a theta divisor and is denoted $\Theta$. 
The datum of $\Theta$ is another way to specify a polarization on $X$. Note that such a divisor $\Theta$ is always reduced. 

A morphism $f\colon (X_1, \mathcal{H}_1) \to (X_2, \mathcal{H}_2)$ between principally polarized abelian varieties is defined as a morphism of $\C$-varieties such that $\mathcal{H}_1 = f^*\mathcal{H}_2$.  
In particular such a morphism is always injective.

For any line bundle $L$ on a principally polarized abelian variety $(X, \Theta)$ we can define a morphism $X \to Pic^0(X)$, $x \mapsto \tau_x^* L \otimes L^{-1}$, where $\tau_x$ is the translation by $x$. The kernel of this morphism is a disjoint union of subtori of $X$; we denote by $K(L)^0$ the irreducible component of the kernel that contains the identity, and by $\pi \colon X \to X/K(L)^0$ the quotient map. Then if $H^0(X,L) \neq 0$, there exists an ample line bundle $\Theta$ on $X/K(L)^0$ such that $\pi^* \Theta = L$ (see \cite[Theorem VI.5.1]{Deb}).

There is a natural notion of product for principally polarized abelian varieties (see \cite[VI.6.1]{Deb}): The product $(X_1, \Theta_1) \times (X_2, \Theta_2)$ is defined as $(X_1 \times X_2, \Theta)$ where $\Theta = p_1^* \Theta_ 1 + p_2^* \Theta_2$, and $p_1, p_2$ are the first and second projections on $X_1 \times X_2$.
We say that a principally polarized abelian variety $(X, \Theta)$ is \emph{indecomposable} if it cannot be written as such a (nontrivial) product.
Then any principally polarized abelian variety $(X, \Theta)$ admits an essentially unique factorization into indecomposable factors, which is directly related to the irreducible components of $\Theta$ (see \cite[VI.9]{Deb} or \cite[Lemma 3.20]{Clemens-Griffiths}):

\begin{theo}\label{thm:factor}
Let $(X, \Theta)$ be a principally polarized abelian variety, and let $D_1, \cdots$, $D_s$ be the irreducible components of $\Theta$. 
For each $k =1,\dots,s$, denote $X_k = X / K(D_k)^0$, $\pi_k\colon X \to X_k$ the projection and $\Theta_k$ a divisor on $X_k$ such that $\pi_k^* \Theta_k = D_k$.
Then 
$$(X, \Theta) = (X_1, \Theta_1) \times \cdots \times (X_s, \Theta_s)$$
is a factorization into indecomposable principally polarized abelian varieties.
Furthermore, if $(X,\Theta) = \prod_{k=1,\dots,s'} (Y_k, \Theta_k')$ is another such factorization, then $s = s'$ and up to a change of index we have $(X_k, \Theta_k)$ isomorphic to $(Y_k, \Theta_k')$ for all $k$.  
\end{theo}

\subsection{Intermediate Jacobians} In this paragraph, we quickly overview the theory of intermediate Jacobians, following  \cite{Clemens-Griffiths} and \cite{Voisin}. \\

Let $X$ be a smooth irreducible projective $\C$-variety. 
For $n\in\mathbb{N}$ such that $2n-1 \le \dim X$, consider the $\mathbb{C}$-vector space
$$V_n = H^{n-1,n}(X) \times H^{n-2,n+1}(X) \times \dots \times H^{0,2n-1}(X).$$
Then the \emph{$n$th intermediate Jacobian} of $X$ is defined as the complex torus:
$$
J_n(X) = V_n/H^{2n-1}(X, \Z),
$$
where $H^{2n-1}(X,\Z)$ is identified to its image via the group morphism $\iota \colon H^{2n-1}(X,\Z)$ $\rightarrow V_n$, which is obtained by composing the canonical morphism $H^{2n-1}(X,\Z)\rightarrow H^{2n-1}(X, \C)$ and the projection $H^{2n-1}(X, \C) \rightarrow V_n$. 
In general $J_n(X)$ has no reason to define an abelian variety.\\

First we specialize to the case where $\dim X = 3$ and $n = 2$. 
Consider the hermitian form defined on $V_2$ by
$$\mathcal H(\alpha, \beta) = i \int_X \alpha \wedge \overline{\beta}.$$
Recall that if $\omega \in H^{1,1}(X)$ is the K\"ahler form induced by the ambient Fubini-Study metric, the Lefschetz operator $L\colon H^k(X) \to H^{k+2}$ is defined by $L(\alpha) = \omega \wedge \alpha$, and the primitive subspace $H^{p,q}_{\text{prim}}(X)$ is defined as the kernel of the restriction of $L^{4 - p - q}$ to $H^{p,q}(X)$. 
One can show that $V_2$ admits a decomposition
$$V_2 = H^{1,2}_{\text{prim}}(X) \times L H^{0,1}(X) \times H^{0,3}(X),$$
and that $\mathcal H$ is positive definite on $H^{1,2}_{\text{prim}}(X)$ and negative definite  on $L H^{0,1}(X) \times H^{0,3}(X)$.
If we assume 
$$H^{0,1}(X) = H^{0,3}(X) = 0,$$
we obtain a principal polarization on $J(X) = J_2(X)$ (see \cite[Theorem 6.32]{Voisin} or \cite[Lemma 3.4]{Clemens-Griffiths}). 
Remark that threefolds which are complete intersections gives examples of varieties satisfying $H^{0,1}(X) = H^{0,3}(X) = 0$.\\

If instead we specialize to the case where $\dim X = 1$ and $n = 1$, we recover the usual Jacobian $J(X) = J_1(X)$ of a curve $X$. 
In this case a canonical principal polarization comes from the hermitian metric defined on $H^{0,1}(X)$ by
$$\mathcal H (\alpha,\beta) = -i \int_X \alpha \wedge \overline{\beta}.$$
Recall the classical result of Torelli which says that we can recover a curve from its principally polarized Jacobian (see \cite[p.~359]{Griffiths-Harris}):

\begin{theo}\label{thm:torelli}
Let $X,X'$ be two smooth irreducible projective $\C$-curves. If their Jacobians $(J(X), \mathcal H)$ and $(J(X'), \mathcal H')$ are isomorphic as principally polarized varieties, then $X$ is isomorphic to $X'$.
\end{theo}

Note also that for a curve $X$, a theta divisor on $J(X)$ is irreducible, because it is the image (up to translation) of the $(g-1)$-symmetric product of $X$ under a natural morphism $X^{(g-1)} \to J(X)$ (see \cite[p.~338]{Griffiths-Harris}). In particular $J(X)$ is indecomposable by Theorem \ref{thm:factor}.  

\subsection{Blow-ups}

For our purposes it is important to understand how the intermediate Jacobian of a threefold is affected under the blow-up of a smooth center.
First, at the level of the cohomology, we have the following fact (this is a particular case of \cite[Theorem 7.31]{Voisin}):

\begin{prop}
Let $X$ be a smooth irreducible projective $\mathbb{C}$-variety of dimension 3, and let $Z \to X$ be the blow-up of a smooth irreducible center $C$.
\begin{enumerate}
\item If $C$ is a point, then we have canonical isomorphisms of $\mathbb{C}$-vector spaces
\begin{align*}
H^{1,1}(Z) &\cong H^{1,1}(X) \oplus \C;\\
H^{p,q}(Z) &\cong H^{p,q}(X) \text{ for } (p,q) \neq (1,1).
\end{align*}
\item If $C$ is a curve, then we have canonical isomorphisms of $\mathbb{C}$-vector spaces
\begin{align*}
H^{1,1}(Z) &\cong H^{1,1}(X) \oplus \C;\\
H^{2,1}(Z) &\cong H^{2,1}(X) \oplus H^{1,0}(C);\\ 
H^{1,2}(Z) &\cong H^{1,2}(X) \oplus H^{0,1}(C);\\ 
H^{p,q}(Z) &\cong H^{p,q}(X) \text{ for other } (p,q).
\end{align*}
\end{enumerate}
\end{prop}

Note in particular that the condition $H^{0,3}(X) = H^{0,1}(X) = 0$ is preserved under such blow-ups. 
Now at the level of the Jacobians we have (see \cite[Lemma 3.11]{Clemens-Griffiths}):

\begin{prop}\label{prop:blowup}
Let $X$ be a smooth irreducible projective $\mathbb{C}$-variety of dimension 3, with $H^{0,3}(X) = H^{0,1}(X) = 0$, and let $Z \to X$ be the blow-up of a smooth irreducible center $C$.
\begin{enumerate}
\item If $C$ is a point or a rational curve, then $J(Z) \cong J(X)$.
\item If $C$ is a non rational curve, then $J(Z) \cong J(X) \times J(C)$.
\end{enumerate}
These are isomorphisms of principally polarized abelian varieties, where $J(X)$, $J(Z)$ and $J(C)$ are equipped with their canonical polarizations defined above.
\end{prop}

\subsection{Consequences} For our main result, Theorem \ref{le theoreme}, we are mainly interested in varieties of dimension 3 over $\C$. In this context, we have the following result.

\begin{prop}
\label{jacobienne}
Let $X$ be a smooth irreducible projective $\mathbb{C}$-variety of dimension 3 whose Hodge numbers $h^{3,0}(X), h^{1,0}(X)$ both vanish. Let $f\colon X\dasharrow X$ be a birational map. Let 
$$
\xymatrix{&Z_1\ar[dl]_{b_1}\ar[dr]^{c_1}&&&&Z_n\ar[dl]_{b_n}\ar[dr]^{c_n}&\\
X&&X_2&\ldots&X_n&&X}
$$
be a weak factorization of $f$. For each $1\leq i\leq n$, let $\mathcal{C}_{b_i}$ (resp. $\mathcal{C}_{c_i}$) be the (finite) set of all the blow-ups $\pi_C$, with center a smooth irreducible non rational $\mathbb{C}$-curve $C$, appearing in the expression of the birational morphisms $b_i$ (resp. $c_i$). 
Then there exists an isomorphism of principally polarized abelian varieties over $\mathbb{C}$
$$
J(X)\times\left(\prod_{\pi_C\in \cup_{i = 1}^n\mathcal{C}_{b_i}} J(C)\right)\simeq J(X)\times\left(\prod_{\pi_{C'}\in \cup_{i = 1}^n\mathcal{C}_{c_i}} J(C')\right).
$$
In particular, we have the following properties:
\begin{enumerate}[a\emph{)}]
\item The sets $\mathcal{C}_b := \cup_{i = 1}^n\mathcal{C}_{b_i}$ and $\mathcal{C}_c := \cup_{i = 1}^n\mathcal{C}_{c_i}$ have same cardinality;
\item There exists a bijection $\sigma\colon \mathcal{C}_b\rightarrow \mathcal{C}_c$ such that, for any $\pi_C$, $\pi_C\in\mathcal{C}_b$, the center $C$ of $\pi_C$ is isomorphic over $k$ to the center of $\sigma(\pi_C)$.
\end{enumerate}
\end{prop}

\begin{proof}

Let us fix $i$, $1\leq i\leq n$, and consider the diagram of blow-ups with irreducible smooth centers
$$
\xymatrix{&Z_i\ar[dl]_{b_i}\ar[dr]^{c_i}&\\
X_i&&X_{i+1}}
$$
coming from the weak factorization of $f$. 
Recall that the morphisms $b_i$ and $c_i$ are finite sequences of blow-ups along irreducible smooth centers. 
By Proposition \ref{prop:blowup}, we know that only blow-ups along non rational curves affect the intermediate Jacobians, and precisely we obtain a canonical isomorphism of principally polarized abelian varieties over $\mathbb{C}$, for each $i$, $1 \le i \le n$:
$$
J(X_i)\times \left(\prod_{\pi_C \in \mathcal C_{b_i}} J(C)\right)
\cong J(Z_i)
\cong J(X_{i+1})\times \left(\prod_{\pi_{C'} \in \mathcal C_{c_i}} J(C')\right).
$$
By taking the product of respectively the left and right-hand sides of these $n$ isomorphisms, and canceling the factors $J(X_2), \dots, J(X_n)$ which appear once on both side (this makes use of Theorem \ref{thm:factor}), we obtain the expected isomorphism of principally polarized abelian varieties over $\mathbb{C}$
\begin{equation}\label{eq:jac}
J(X)\times\left(\prod_{\pi_C\in \mathcal{C}_{b}} J(C)\right)\simeq J(X)\times\left(\prod_{\pi_{C'}\in \mathcal{C}_{c}} J(C')\right).
\end{equation}

Then assertion $a$) follows from (\ref{eq:jac}) and Theorem \ref{thm:factor}; and assertion $b$) follows from $a$) and Torelli's Theorem \ref{thm:torelli}.
\end{proof}

\begin{coro}
\label{corollaire jacobienne}
Let $X$ be a smooth irreducible projective $\mathbb{C}$-variety of dimension 3 whose Hodge numbers $h^{3,0}(X), h^{1,0}(X)$ both vanish. Let $f\colon X\dasharrow X$ be a birational map. Let $$
\xymatrix{&Z_1\ar[dl]_{b_1}\ar[dr]^{c_1}&&&&Z_n\ar[dl]_{b_n}\ar[dr]^{c_n}&\\
X&&X_2&\ldots&X_n&&X}
$$
be a weak factorization of $f$. Let $\tilde{\mathcal{C}}_b$ (resp. $\tilde{\mathcal{C}}_c$) be the (finite) set of all of the blow-ups with center a point or a rational curve, appearing in the expression of any of the birational morphism $b_i$ (resp. $c_i$), $1\leq i\leq n$.
Then there exists a bijection $\tilde{\mathcal{C}}_b\rightarrow \tilde{\mathcal{C}}_c$.
\end{coro}

\begin{proof}
This follows from Proposition \ref{jacobienne}, which says that $\mathcal C_b$ and $\mathcal C_c$ have the same cardinality, and from Corollary \ref{piece weak}, which says that $\mathcal C_b \cup \tilde{\mathcal{C}}_b$ and $\mathcal C_c \cup \tilde{\mathcal{C}}_c$ have the same cardinality.
\end{proof}

\begin{coro}
\label{corollaire jacobienne2}
Let $X$ be a smooth irreducible projective $\mathbb{C}$-variety of dimension 3 whose Hodge numbers $h^{3,0}(X), h^{1,0}(X)$ both vanish. Let $f\colon X\dasharrow X$ be a birational map. Let 
$$
\xymatrix{&Z_1\ar[dl]_{b_1}\ar[dr]^{c_1}&&&&Z_n\ar[dl]_{b_n}\ar[dr]^{c_n}&\\
X&&X_2&\ldots&X_n&&X}
$$
be a weak factorization of $f$. Let $C$ be a smooth irreducible projective curve $C$ of genus $g(C)\geq 1$. Then the number of blow-ups, with center isomorphic to $C$, appearing in the $b_i$, $1\leq i\leq n$, is equal to the number of blow-ups, with center isomorphic to $C$, appearing in the $c_i$, $1\leq i\leq n$. 
\end{coro}

\begin{proof}
The result follows from Proposition \ref{jacobienne}, from the unicity of the factorization into indecomposable factors (Theorem \ref{thm:factor}) and from Torelli's Theorem \ref{thm:torelli}. 
\end{proof}

\section{The main theorem}
\label{MAIN}

\subsection{Preliminary results and remarks}

We assume that $k$ is algebraically closed of characteristic zero. Let $X$ be a $k$-variety, of dimension $d\geq 0$, and let $f\colon X\dasharrow X$ be a map that induces an isomorphism $U\cong V$ on dense open subvarieties $U,V$ of $X$.

\subsubsection{Components of maximal dimension}\label{uniregle} By Lemma \ref{dimension}, we see that the $k$-varieties $(X\setminus U)_{\mathrm{red}}$ and $(X\setminus V)_{\mathrm{red}}$ have automatically the same dimension, and the same number of irreducible components of maximal dimension.\\

If $U$ is not the complement of the  exceptional set of $f$, the variety $(X\setminus U)_{\mathrm{red}}$ could contain non-uniruled irreducible components of maximal dimension (see Lemma \ref{unireglage}). But then, the set of non-uniruled irreducible components of maximal dimension of $(X\setminus U)_{\mathrm{red}}$ is automatically in bijection with that of of $(Y\setminus U)_{\mathrm{red}}$, by \cite[Theorem 2]{Liu-Sebag}. In particular, if $U_0$ is the complement of the  exceptional set of $f$, and $V_0:=f(U_0)$, then the set of all the irreducible components of maximal dimension in $U_0 \setminus U$ is in bijection with that of $V_0 \setminus V$.

\subsubsection{State of the art}\label{dimension 2} 

We gather here what was known about Conjecture \ref{la conjecture} before this work.

\begin{enumerate}[$a$)]

\item Let $X$ be a $k$-variety of dimension $d\geq 0$. Assume that $X$ contains only finitely many rational curves. It follows from \cite[Theorem 5]{Liu-Sebag} that Conjecture \ref{la conjecture} holds true. (See also \cite[\S 13]{HK}.)

\item Assume that the dimension of $(X\setminus U)_{\mathrm{red}}$ is at most 1, then Conjecture \ref{la conjecture} holds true for $X$ in that case by Proposition \ref{cas des courbes} (see also \cite[Proposition 6]{Liu-Sebag}).

This remark implies in particular that Conjecture \ref{la conjecture} holds true for any $k$-variety $X$ of dimension 2. 

\item Assume that $X$ is projective, of dimension at least 3, that the projective $k$-varieties $(X\setminus U)_{\mathrm{red}}$, $(X\setminus V)_{\mathrm{red}}$ are smooth surfaces over $k$. Then Conjecture \ref{la conjecture} holds true for $X$ in that case by \cite[Theorem 4]{Liu-Sebag}. 

\item Assume that the $k$-varieties $(X\setminus U)_{\mathrm{red}}$, $(X\setminus V)_{\mathrm{red}}$ have dimension 2, and that their irreducible components of maximal dimension are rational, then Conjecture \ref{la conjecture} holds true for $X$ in that case by \cite{Sebag-PAMS}. One can prove directly this statement by Lemma \ref{dimension} and Proposition \ref{cas des courbes}.

\item  Let $X$ be a smooth irreducible proper $k$-variety with a birational map $f\colon X\dasharrow X$ which admits a weak factorization obtained by blowing-up smooth irreducible projective rational $k$-surfaces, smooth irreducible projective rational $k$-curves or points. Then, if the dimension of $(X\setminus U)_{\mathrm{red}}$ is at most 2, the $k$-varieties $(X\setminus U)_{\mathrm{red}}$ and $(X\setminus V)_{\mathrm{red}}$ are piecewise isomorphic. Indeed, the assumptions implies that each irreducible component of dimension 2 of $(X\setminus U)_{\mathrm{red}}$ or $(X\setminus V)_{\mathrm{red}}$ is rational. We conclude by Lemma \ref{dimension} and Proposition \ref{cas des courbes}.
\end{enumerate}

\subsection{Proof of Main Theorem}

In this paragraph, we prove the following theorem which is the main result of this article.

\begin{theo}
\label{le theoreme}
Let $X$ be a smooth irreducible projective $\C$-variety of dimension 3, whose Hodge numbers $h^{1,0}(X), h^{3,0}(X)$ both vanish. Let $f\colon X\dasharrow X$ be a birational map that induces an isomorphism $U\cong V$ on dense open subvarieties $U,V$ of $X$. Then the $\mathbb{C}$-varieties $(X\setminus U)_{\mathrm{red}},(X\setminus V)_{\mathrm{red}}$ of $X$ are piecewise isomorphic.
\end{theo}

\begin{proof}

First we observe that it is sufficient to prove that the irreducible components of dimension 2 of $(X\setminus U)_{\mathrm{red}}$ and $(X\setminus V)_{\mathrm{red}}$ are mutually birationally equivalent. 
Indeed, if there exist an open dense subscheme $U_1$ of $(X\setminus U)_{\mathrm{red}}$, an open dense subscheme $V_1$ of $(X\setminus V)_{\mathrm{red}}$, and an isomorphism $f_2\colon U_1 \rightarrow V_1$ of $\mathbb{C}$-varieties, then Proposition \ref{cas des courbes} produces a piecewise isomorphism $(f_0,f_1)$ from $((X\setminus U)\setminus U_1)_{\mathrm{red}}$ to $((X\setminus V)\setminus V_1)_{\mathrm{red}}$. The datum of $(f_0, f_1,f_2)$ is a piecewise isomorphism between $(X\setminus U)_{\mathrm{red}}$ and $(X\setminus V)_{\mathrm{red}}$. From now, we will only prove that there exists a bijection $\sigma$ from the set of all the irreducible components of dimension 2 of $(X\setminus U)_{\mathrm{red}}$ to those of $(
X\setminus V)_{\mathrm{red}}$ such that, for any such a component $F$ of $(X\setminus U)_{\mathrm{red}}$, the $\C$-varieties $F$ and $\sigma(F)$ are birationally equivalent.\\

Now the proof proceeds from a precise analysis of the weak factorization of $f$, via Proposition \ref{jacobienne}.
First, by \S \ref{uniregle} we can assume without loss in generality that $U$ is the largest open subset of $X$ on which the restriction of $f$ is an isomorphism. 
In particular all components of dimension 2 in $(X\setminus U)_{\mathrm{red}}$ and $(X\setminus V)_{\mathrm{red}}$ are ruled.

Let 
$$
\xymatrix{&Z_1\ar[dl]_{b_1}\ar[dr]^{c_1}&&&&Z_n\ar[dl]_{b_n}\ar[dr]^{c_n}&\\
X&&X_2&\ldots&X_n&&X}
$$
be a weak factorization of $f$. Let $\eta$ be the generic point of an irreducible component of dimension 2 of $(X\setminus U)_{\mathrm{red}}$. Thanks to our hypothesis, we know that there exists a smooth irreducible projective curve $C_\eta$ such that the residue field $\kappa(\eta)$ of $X$ at $\eta$ is isomorphic to $\kappa(C_\eta)(T)$, the quotient field of the ring of polynomials in $T$ with coefficients in the functions field of $C_\eta$. We also know that such a point must be contracted by a morphism $c_i$, because of our assumption on $U$ and the dimension of $\eta$. If we replace $c_i$ by $b_i$, the generic points $\eta'$ of the irreducible components of dimension 2 of $(X\setminus V)_{\mathrm{red}}$ verify the same properties.\\

So, it is sufficient to prove that for any smooth irreducible projective curve $C$ of genus $g(C)\geq 1$, the number of generic points $\eta$ in $(X\setminus U)_{\mathrm{red}}$ with residue field isomorphic to $\kappa(C)(T)$ is equal to the number of generic points $\eta'$ in $(X\setminus V)_{\mathrm{red}}$ with $\kappa(\eta')\simeq \kappa(C)(T)$. 
Indeed, by Lemma \ref{dimension}, we conclude that the number of rational irreducible components in $(X\setminus U)_{\mathrm{red}}$ will coincide with that in $(X\setminus V)_{\mathrm{red}}$.\\

Let $C$ be a smooth irreducible projective curve $C$ of genus $g(C)\geq 1$. By Corollary \ref{corollaire jacobienne2}, we know that in the weak factorization of $f$, the number of blow-ups with center $C$ appearing in the $b_i$ is equal to those appearing in the $c_i$. 

We define now a numbering on the set of the blow-ups, with center $C$, that appear in the $b_i$ as follows. In $b_1$, consider the first blow-up of this kind. If the strict transform of its exceptional divisor is the exceptional divisor of a blow-up appearing in a $c_j$, we denote it by $b(1)$, and by $c(1)$ the corresponding blow-up in $c_j$. If not, consider the second blow-up in $b_1$ and so on. When we have tested all the blow-ups of this kind in $b_1$, we apply the same procedure to $b_2$ and so on. Then, if it exists, we have found the first numbered blow-up $b(1)$, with center $C$. It determines ipso facto $c(1)$ that is also a blow-up with center $C$. At this step, we search for the second numbered blow-up $b(2)$ such that the strict transform of its exceptional divisor is $k$-isomorphic to the exceptional divisor of a blow-up appearing in a $c_{j'}$, distinct from the blow-up $c(1)$ by construction, and so on.

At the end, we have numbered all the blow-ups with center $C$ appearing in the $b_i$, by $b(\ell)$, $1\leq \ell \leq s$, except a finite number of them which are those whose strict transform are the irreducible components of $(X\setminus V)_{\mathrm{red}}$ birationally equivalent to $C\times_{\mathbb{C}} \mathbb{P}^1_{\mathbb{C}}$. Then we can split the formula in Proposition \ref{jacobienne} as follows
\begin{multline} \label{eq:jacob}
J(X)\times\left(\prod_{\tilde{C}\not=C} J(\tilde{C})\right)\times\left(\prod_{1\leq \ell\leq s} J(C)\right)\times \left(\prod_{\ell\in \Gamma} J(C)\right) \\
\simeq J(X)\times\left(\prod_{\tilde{C}'\not=C} J(\tilde{C}')\right)\times\left(\prod_{1\leq \ell\leq s} J(C)\right)\times \left(\prod_{\ell'\in \Gamma'} J(C)\right)
\end{multline}
where $\Gamma$ is a set in bijection with the (finite) set of irreducible component in $(X\setminus V)_{\mathrm{red}}$, birationally equivalent to $C\times_{\mathbb{C}}\mathbb{P}^1_{\mathbb{C}}$, and where, by symmetry, $\Gamma'$  is a (finite) set in bijection with the set of irreducible component in $(X\setminus U)_{\mathrm{red}}$, birationally equivalent to $C\times_{\mathbb{C}}\mathbb{P}^1_{\mathbb{C}}$. By unicity of the factorization of a principally polarized abelian variety (Theorem \ref{thm:factor}), we deduce from (\ref{eq:jacob}) that $\Gamma$ is in bijection with $\Gamma'$. In other words the numbers of irreducible components birationally equivalent to $C\times_{\mathbb{C}} \mathbb{P}^1_{\mathbb{C}}$ in $(X\setminus U)_{\mathrm{red}}$ and $(X\setminus V)_{\mathrm{red}}$ are equal.
\end{proof}

\begin{rema} 
\label{Gen thm}
The arguments of the proof of Theorem \ref{le theoreme} also imply the following statement. Let $X,Y$ be smooth projective $\mathbb{C}$-varieties of dimension 3 (possibly reducible), with $h^{3,0}(X)=h^{1,0}(X)=h^{3,0}(Y)=h^{1,0}(Y)=0$, and such that there exists an isomorphism $J(X)\cong J(Y)$ of principally polarized abelian varieties over $\mathbb{C}$. If there exists a birational map $f\colon X\dasharrow Y$ that induces an isomorphism of $\mathbb{C}$-schemes $U\rightarrow V$ between dense open subvarieties $U$ of $X$ and $V$ of $Y$, then the $\mathbb{C}$-varieties $(X\setminus U)_{\mathrm{red}}$ and $(Y\setminus V)_{\mathrm{red}}$ are piecewise isomorphic. 
\end{rema}

\begin{coro}
Let $X$ be a smooth irreducible projective $\mathbb{C}$-variety, of dimension 3, whose Hodge numbers $h^{3,0}(X),h^{1,0}(X)$ both vanish. Let $f\colon X\dasharrow X$ be a birational map that induces an isomorphism of $\mathbb{C}$-schemes $U\cong V$ on dense open subvarieties $U,V$ of $X$. Let $A$ be a ring, and $\mu\colon \KZ\rightarrow A$ be a motivic measure. Then
$$
\mu([(X\setminus U)_{\mathrm{red}}])=\mu([(X\setminus V)_{\mathrm{red}}]).
$$
In particular, the projective $\mathbb{C}$-varieties $(X\setminus U)_{\mathrm{red}},(X\setminus V)_{\mathrm{red}}$ have the same Hodge--Deligne polynomial, the same Poincar\'e polynomial, the same (topological) Euler characteristic, and $\mathrm{SB}([(X\setminus U)_{\mathrm{red}}])=\mathrm{SB}([(X\setminus V)_{\mathrm{red}}])$.
\end{coro}

\begin{proof}
By Theorem \ref{le theoreme}, $[(X\setminus U)_{\mathrm{red}}]=[(X\setminus V)_{\mathrm{red}}]$ in $K_0(Var_\mathbb{C})$. So, by the definition of a motivic measure, $\mu([(X\setminus U)_{\mathrm{red}}])=\mu([(X\setminus V)_{\mathrm{red}}])$. Recall that the Hodge--Deligne realization morphism, the Poincar\'e realization morphism, the (topological) Euler characteristic, the morphism $\mathrm{SB}$ are examples of motivic measure (see \S \ref{realisation de Hodge-Deligne}).
\end{proof}

\section{Examples}
\label{EXPLE}

One particularly rich source of examples for Theorem \ref{le theoreme} is when we consider $X$ equal to the projective space. We give below three explicit examples in this setting. 
The first one is in dimension 2, and illustrates the fact that the exceptional loci of a transformation and its inverse might be quite far from being isomorphic, despite the fact that they are piecewise isomorphic.

The two last examples are in dimension 3, and as such are direct illustrations of Theorem \ref{le theoreme}. 

\begin{exam}[see {\cite[Example 2.1.14]{AC}}]
Consider the plane Cremona map $f\colon \mathbb{P}^2_\mathbb{C}\dasharrow \mathbb{P}^2_\mathbb{C}$ defined by the datum of the three following homogeneous polynomials, expressed in homogeneous coordinates $x,y,z$:
\begin{align*}
F&= (x^3-yz(x+y))(x^2-yz)(x+y),\\
G&=x^2(x^2-yz)(x+y)^2,\\
H&=x^3(x^3-yz(x+y)).
\end{align*}
The determinant of the Jacobian matrix $\begin{pmatrix} \partial_x F & \partial_y F & \partial_z F \\ \partial_x G & \partial_y G & \partial_z G \\ \partial_x H &\partial_y H & \partial_z H \end{pmatrix}$ is equal to
$$ \text{Jac}(f) = -6x^3y^2z^2(x+y)^3(x+z)^3(y-z)^2.$$
We deduce the (reduced) equation of the contracted curves:
$$ xyz(x+y)(x+z)(y-z) = 0.$$
We see that the exceptional locus is equal to the disjoint union of
\begin{itemize}
\item Seven points $[0:0:1]$, $[0:1:0]$, $[1:0:0]$, $[-1:1:0]$, $[-1:1:1]$, $[0:1:1]$, $[-1:0:1]$;
\item Six curves isomorphic to $ \mathbb{P}^1_\mathbb{C}$ minus three points.
\end{itemize}
The inverse of $f$ has the form
\begin{multline*}
f^{-1} \colon [x:y:z] \dashrightarrow \\
[-xyz(x+y)(x+z)(y-z):-x^2y(x+z)(y-z)^2:z^2(y+x)^2(x^2-yz)]
\end{multline*}
with Jacobian determinant
$$ \text{Jac}(f^{-1}) = 6x^4y(x+y)^2(x^2-yz)(x^3-y^2z-xyz)(x^2y+2x^3-y^2z-xyz).$$
So the exceptional locus of $f^{-1}$ has six irreducible components described as follows with their reduced structures: three lines, a conic and two nodal cubics.
We can write this locus as the disjoint union of
\begin{itemize}
\item Three points $[0:0:1]$, $[0:1:0]$, and $[-1:1:1]$;
\item One curve isomorphic to $\mathbb{P}^1_\mathbb{C}$ minus one point (line $y =0$), two curves isomorphic to $\mathbb{P}^1_\mathbb{C}$ minus two points (corresponding to the two other lines), and three other curves isomorphic to $\mathbb{P}^1_\mathbb{C}$ minus three points (corresponding to the conic and the nodal cubics).
\end{itemize}
We see that we can realize a piecewise isomorphism between the exceptional locus of $f$ and $f^{-1}$ by choosing an arbitrary additional point on each one of the lines $x = 0$ and $x+y = 0$, and two arbitrary additional points on the line $y = 0$. 
\end{exam}

\begin{exam}
We give now an example in dimension 3. Consider a general transformation of $\mathbb{P}^3_\mathbb{C}$ of bidegree $(2,4)$, meaning that $f$ has degree 2 and $f^{-1}$ has degree 4: see \cite{PRV}, where this transformation is denoted by gen$^{[4]}$.
One representant (up to left and right composition by linear automorphisms) is for instance 
$$f\colon [x:y:z:t] \dashrightarrow [yz: xz: xy: t(x+y+z)]$$
with inverse
$$f^{-1} \colon  [x:y:z:t] \dashrightarrow [(yz + xz + xy)yz: (yz + xz + xy)xz: (yz + xz + xy)xy: xyzt].$$ 
We compute the Jacobian determinants of $f$ and $f^{-1}$:
\begin{align*}
\text{Jac}(f)  &= xyz(x+y+z); \\
\text{Jac}(f^{-1})  &= x^2y^2z^2(yz + xz + xy)^3.
\end{align*}
The exceptional set of $f$ is equal to the disjoint union of:
\begin{itemize}
\item The point $p = [0:0:0:1]$;
\item Six curves isomorphic to $\mathbb{P}^1_\mathbb{C} \setminus \{p \}$;
\item Four surfaces isomorphic to $\mathbb{P}^2_\mathbb{C}$ minus three lines with a common point. 
\end{itemize}
On the other hand the exceptional set of $f^{-1}$ is equal to the disjoint union of:
\begin{itemize}
\item The point $p = [0:0:0:1]$;
\item Three curves isomorphic to $\mathbb{P}^1_\mathbb{C} \setminus \{p \}$;
\item Three surfaces isomorphic to $\mathbb{P}^2_\mathbb{C}$ minus two lines;
\item A surface isomorphic to a quadratic cone minus three rules, which is also isomorphic to $\mathbb{P}^2_\mathbb{C}$ minus three lines with a common point.
\end{itemize}
By choosing a general line passing through $p$ in each plane $x=0$, $y=0$, $z=0$, we can recover a stratification of $\text{Exc}(f^{-1})$ which realizes a piecewise isomorphism with  $\text{Exc}(f)$.  
\end{exam}

\begin{exam}\label{exple:cubic}
The following example is probably the simplest birational self-map $g$ of $\cpt$ such that 
\begin{enumerate}[(a)]
\item $\Exc(g)$ and $\Exc(g^{-1})$ are not isomorphic;
\item $\Exc(g)$ contains an irreducible component which is not a rational surface.
\end{enumerate}

Consider the smooth plane cubic $C \subset \cpd$ with homogeneous equation
$$q_3 (x,y,z) := y^2z - x^3 +xz^2 = 0.$$
Then 
$$f\colon (x,y,z) \dashrightarrow \left(\frac{xq_3}{x^3}, \frac{yq_3}{x^3},\frac{zq_3}{x^3} \right)$$
is a birational transformation of $\cpt$ (expressed in an affine chart $\C^3$), with inverse
$$f^{-1}\colon (x,y,z) \dashrightarrow \left( \frac{x^4}{q_3}, \frac{yx^3}{q_3}, \frac{zx^3}{q_3} \right).$$ 

Note that we could construct a similar example starting from any irreducible plane curve $C$ of arbitrary degree (see \cite{Pan}). The associated transformation contracts a surface birational to $C \times \cpo$: We see that transformations contracting only rational surfaces are far from being the rule.

The transformations $f$ and $f^{-1}$ have same exceptional loci, so we now compose $f$ with the cubo-cubic involution
$$\sigma\colon (x,y,z) \dashrightarrow \left( \frac{1}{x}, \frac{1}{y}, \frac{1}{z}\right)$$
to get a more interesting example.
We set $g =  \sigma \circ f$, so $g^{-1} = f^{-1} \circ \sigma$, and explicitly
\begin{align*}
g &\colon (x,y,z) \dashrightarrow \left(\frac{x^3}{xq_3}, \frac{x^3}{yq_3},\frac{x^3}{zq_3} \right), \\
g^{-1} &\colon (x,y,z) \dashrightarrow \left(\frac{y^2z^2}{xq_4},\frac{xyz^2}{xq_4},\frac{xy^2z}{xq_4} \right),
\end{align*}
where $q_4 (x,y,z) := x^3z-y^2z^2+x^2y^2$ satisfies
$q_3\left( \frac1x, \frac1y, \frac1z \right) = \dfrac{q_4(x,y,z)}{x^3y^2z^2}.$

In homogeneous coordinates we have
\begin{align*}
g &\colon [x:y:z:t] \dashrightarrow [x^2yzt:x^3zt:x^3yt:yzq_3], \\
g^{-1} &\colon [x:y:z:t] \dashrightarrow [y^2z^2t: xyz^2t: xy^2zt: xq_4].
\end{align*}

We compute the Jacobian determinants:
\begin{align*}
\text{Jac}(g) &= -5x^7y^2z^2t^2q_3; \\
\text{Jac}(g^{-1}) &= -5x^2y^4z^4t^2q_4.
\end{align*}

The exceptional locus of $g$ is the union of the 4 coordinates planes and of the cone over the smooth elliptic curve $\{t = q_3(x,y,z) = 0\}$, isomorphic to $C$.
Similarly, the exceptional locus of $g$ is the union of the 4 coordinates planes and of the cone over the singular quartic curve $C' = \{t = q_4(x,y,z) = 0\}$.

\begin{table}[t]
$$\begin{array}{l|l|l}
&\Exc(g)&\Exc(g^{-1}) \\ \hline
\text{Points} & \bullet 6 \text{ points } p_0, \dots, p_5  & \bullet 4 \text{ points } p_0, \dots, p_3 \\ \hline
\text{Curves} & \bullet 7 \text{ components} \cong \C^*: & \bullet 6 \text{ components} \cong \C^*: \\ 
&\text{lines } (p_0, p_i), i = 1, \dots 5, & \text{Lines } (p_0, p_i), i = 1, \dots 3,   \\ 
&\text{and} (p_1,p_2), (p_2,p_3).  & \text{and} (p_1,p_2), (p_2,p_3), (p_1,p_3).   \\ 
&\bullet 1 \text{ component } \cong \cpo \setminus 4 \text{ points}: &   \\ 
&\text{line } (p_1, p_3) \text{ minus } \{ p_1, p_3, p_4, p_5\} & \bullet 1 \text{ component } \cong C \setminus 4 \text{ points}: \\ 
&\bullet 1 \text{ component } \cong C \setminus 4 \text{ points}: & q_4 = 0 \text{ minus } \{ p_1, p_2, p_3\}\\ 
&q_3 = 0 \text{ minus } \{ p_2, p_3, p_4, p_5\} & \\ \hline
\text{Surfaces} &\bullet 2 \text{ components } \cong  \mathbb{A}^2_\C \text{ minus } 2 \text{ lines}: &  \bullet 3 \text{ components } \cong \mathbb{A}^2_\C  \text{ minus } 2 \text{ lines}:\\
 & \text{come from } x = 0 , z = 0. & \text{come from } x = 0 , y = 0, z = 0. \\
 & \bullet 1 \text{ component } \cong \mathbb{A}^2_\C \text{ minus } 4 \text{ lines}: & \\
& \text{comes from } y = 0. & \\
& \bullet 1 \text{ component } \cong (\C^*)^2 \setminus \{ q_3 = 0\}: & \bullet 1 \text{ component } \cong (\C^*)^2 \setminus \{ q_4 = 0\}: \\
& \text{comes from } t = 0. & \text{comes from } t = 0. \\
 & \bullet \text{Cone over } C \setminus \{ 4 \text{ points}\} & \bullet \text{Cone over } C' \setminus \{ 3 \text{ points}\} \\ \hline 
\end{array} $$
\caption{Stratification of exceptional loci in Example \ref{exple:cubic}.}
\label{table}
\end{table}

We set
\begin{align*}
p_0 &= [0:0:0:1], & p_1 &= [0:0:1:0], \\
p_2 &= [0:1:0:0], & p_3 &= [1:0:0:0], \\
p_4 &= [1:0:1:0], & p_5 &= [-1:0:1:0],
\end{align*}
and we note that in the plane $t = 0$, the singular locus of the curve $xyzq_3 = 0$ is equal to the five points $\{p_1, \dots, p_5\}$, and the  singular locus of $xyzq_4 = 0$ is equal to the three points $\{p_1, p_2, p_3\}$.

Observe also that there exists a birational morphism $\pi\colon C \to C'$, since $C$ is isomorphic to the desingularisation of $C'$, and $\pi$ induces an isomorphism between $C' \setminus \{p_1, p_2,p_3\} $ and $C$ minus four points. Indeed $p_2$ is an ordinary double point for $C'$, hence $\pi^{-1}(p_2)$ is equal to two points, and $p_1, p_3$ are respectively a smooth and a cuspidal double point for $C'$ hence their preimage by $\pi$ both consist of one single point.

We have natural stratifications of the exceptional loci, as detailled in Table \ref{table}.

Note that by construction the quadratic involution $[x:y:z] \dashrightarrow [yz:xz:xy]$ realizes an isomorphism between $(\C^*)^2 \setminus \{ q_3 = 0\}$ and $(\C^*)^2 \setminus \{ q_4 = 0\}$. 
We can realize a piecewise isomorphism by picking two well-chosen new rules in $\{y = 0\} \subset \text{Exc}(g^{-1})$, and at most four well-chosen new rules in both cones over the elliptic curve in $\text{Exc}(g)$ and $\text{Exc}(g^{-1})$.

\end{exam}

\section{Further comments}
\label{COM}

In this section, we assume for simplicity that $k$ is an algebraically closed field of characteristic zero. As we have shown in the introduction, Conjecture \ref{la conjecture} is linked to a stronger problem: Question \ref{Larsen-Lunts} (see \cite[Question 1.2]{Larsen-Lunts}).

The strategy of resolution of Question \ref{Larsen-Lunts} in \cite{Liu-Sebag} is stopped by different problems of cancellation (geometric and motivic). Besides, some preexisting geometric problems of cancellation can be formulated, or interpreted via Question \ref{Larsen-Lunts}, in the Grothendieck ring of varieties.

We list here some problems related to these cancellation questions.

\subsection{Zero divisor} Is $\mathbb{L}$ a zero divisor in $\KZ$ ? That question corresponds to a ``piecewise'' cancellation problem. Let $\mathfrak{a}\in \KZ$ such that $\mathbb{L}\cdot\mathfrak{a}=0$. One can represent the element $\mathfrak{a}$ as a formal some of classes of the following type
$$
\mathfrak{a}=\sum_{i=1}^m[X_i]-\sum_{i=1}^n[Y_i],
$$
where all the $X_i,Y_i$ are smooth irreducible projective $k$-varieties. Now, if $\mathbb{L}$ is not a zero divisor in $\KZ$ then $\mathfrak{a}=0$. Thus we are studying the validity of the following formula
$$
\mathbb{L}\cdot\left([\sqcup_{i=1}^m X_i]\right)=\mathbb{L}\cdot\left([\sqcup_{i=1}^n Y_i]\right)~~\Rightarrow ~~[\sqcup_{i=1}^m X_i]=[\sqcup_{i=1}^n Y_i].
$$

Note that, if $\mathbb{L}$ is a regular element of $\KZ$, we deduce formally that Question \ref{Larsen-Lunts} admits a positive answer for any pair of $k$-varieties of dimension 2 (see \cite{Sebag-PAMS}). In that case, Conjecture \ref{la conjecture} would be true for any $k$-varieties of dimension 3.

\subsection{Rationality} Is it possible to characterize rationality in $\KZ$? That question admits a positive answer in dimension at most 2, by \cite[Proposition 7,Lemma 12]{Liu-Sebag}. In particular, is it possible that a stably rational (or unirational, i.e., dominated by $\mathbb{P}^{\mathrm{dim}(X)}_k$), non-rational, smooth proper $k$-variety $X$, of dimension at least 3, has the same class in $\KZ$ than a rational smooth proper $k$-variety (of the same dimension)? \\

\if false
In \cite{Liu-Sebag}, the second author has obtained  elements of answer to Question \ref{Larsen-Lunts} thanks to a birational cancellation lemma for any pair of non-uniruled integral $k$-varieties. But this kind of strategy is stopped exactly because of the existence of stably rational, non-rational, smooth projective $k$-varieties (whenever its dimension is at least 3, see \cite{Sansuc}). \\
\fi

\subsection{Cancellation} In affine geometry, Zariski's Cancellation Problem can be formulated as follows. If $X,Y$ is a pair of affine $k$-varieties such that $X\times_k \mathbb{A}^1_k,Y\times_k \mathbb{A}^1_k$ are $k$-isomorphic, is it true that $X,Y$ are $k$-isomorphic? One knows that some further assumptions are necessary to hope for a positive answer: indeed there exist counterexamples (e.g., Danielewski's counterexample) formed by a couple $X,Y$ of non-isomorphic affine $k$-varieties such that $X\times_k \mathbb{A}^1_k,Y\times_k \mathbb{A}^1_k$ are $k$-isomorphic.\\

We can formulate the question in $\KZ$. We are interested in couples of affine $k$-varieties $X,Y$ such that $\mathbb{L}[X]=\mathbb{L}[Y]$. If $\mathbb{L}$ is a regular element of $\KZ$, we conclude that any such a couple verifies $[X]=[Y]$. If Question \ref{Larsen-Lunts} admits a positive answer, then $X,Y$ must be piecewise isomorphic. Up to our knowledge, and up to now, all the counterexamples to Zariski's Cancellation Problem (in affine geometry) do not contradict this conclusion. Moreover, it follows from \cite[Theorem 2, Proposition 6]{Liu-Sebag}, that any counterexample which does not verify that piecewise condition must be of dimension at most 3.\\

So it seems to us interesting to study the following question (see \cite{Sebag-PAMS}). Does there exist a couple of integral affine $k$-varieties $X,Y$, of dimension $d\geq 3$, such that $X\times_k \mathbb{A}^1_k,Y\times_k \mathbb{A}^1_k$ are $k$-isomorphic, and $X,Y$ not birationally equivalent?

\bibliographystyle{abbrv}
\bibliography{biblio}

\begin{thebibliography}{10}

\bibitem{Wlo}
D.~Abramovich, K.~Karu, K.~Matsuki, and J.~W{\l}odarczyk.
\newblock Torification and factorization of birational maps.
\newblock {\em J. Amer. Math. Soc.}, 15(3):531--572, 2002.

\bibitem{AC}
M.~Alberich-Carrami{\~n}ana.
\newblock {\em Geometry of the plane {C}remona maps}, volume 1769 of {\em
  Lecture Notes in Mathematics}.
\newblock Springer-Verlag, Berlin, 2002.

\bibitem{Clemens-Griffiths}
C.~H. Clemens and P.~A. Griffiths.
\newblock The intermediate {J}acobian of the cubic threefold.
\newblock {\em Ann. of Math. (2)}, 95:281--356, 1972.

\bibitem{Deb}
O.~Debarre.
\newblock {\em Tores et vari\'et\'es ab\'eliennes complexes}, volume~6 of {\em
  Cours Sp\'ecialis\'es [Specialized Courses]}.
\newblock Soci\'et\'e Math\'ematique de France, Paris, 1999.

\bibitem{Denef-Loeser}
J.~Denef and F.~Loeser.
\newblock Germs of arcs on singular algebraic varieties and motivic
  integration.
\newblock {\em Invent. Math.}, 135(1):201--232, 1999.

\bibitem{Griffiths-Harris}
P.~Griffiths and J.~Harris.
\newblock {\em Principles of algebraic geometry}.
\newblock Wiley Classics Library. John Wiley \& Sons Inc., New York, 1994.
\newblock Reprint of the 1978 original.

\bibitem{Gromov}
M.~Gromov.
\newblock Endomorphisms of symbolic algebraic varieties.
\newblock {\em J. Eur. Math. Soc. (JEMS)}, 1(2):109--197, 1999.

\bibitem{Groth_dR}
A.~Grothendieck.
\newblock On the de {R}ham cohomology of algebraic varieties.
\newblock {\em Inst. Hautes \'Etudes Sci. Publ. Math.}, (29):95--103, 1966.

\bibitem{HK}
E.~Hruchovski and D.~Kazhdan.
\newblock {\em Integration in valued fields, \emph{In: Algebraic Geometry and
  Number Theory}}, volume 253 of {\em Progr. Math.}
\newblock Birk\"auser Boston, Boston, 2006.

\bibitem{Ivorra-Sebag}
F.~Ivorra and J.~Sebag.
\newblock G\'eom\'etrie alg\'ebrique par morceaux, {$K$}-\'equivalence et
  motifs.
\newblock {\em To appear in L'Ens. Math. (2012)}.

\bibitem{Larsen-Lunts}
M.~Larsen and V.~A. Lunts.
\newblock Motivic measures and stable birational geometry.
\newblock {\em Mosc. Math. J.}, 3(1):85--95, 2003.

\bibitem{Liu}
Q.~Liu.
\newblock {\em Algebraic geometry and arithmetic curves}, volume~6 of {\em
  Oxford Graduate Texts in Mathematics}.
\newblock Oxford University Press, Oxford, 2002.
\newblock Translated from the French by Reinie Ern{\'e}, Oxford Science
  Publications.

\bibitem{Liu-Sebag}
Q.~Liu and J.~Sebag.
\newblock The {G}rothendieck ring of varieties and piecewise isomorphisms.
\newblock {\em Math. Z.}, 265(2):321--342, 2010.

\bibitem{Pan}
I.~Pan.
\newblock Une remarque sur la g\'en\'eration du groupe de {C}remona.
\newblock {\em Bol. Soc. Brasil. Mat. (N.S.)}, 30(1):95--98, 1999.

\bibitem{PRV}
I.~Pan, F.~Ronga, and T.~Vust.
\newblock Transformations birationnelles quadratiques de l'espace projectif
  complexe \`a trois dimensions.
\newblock {\em Ann. Inst. Fourier (Grenoble)}, 51(5):1153--1187, 2001.

\bibitem{Sebag-PAMS}
J.~Sebag.
\newblock Variations on a question of {L}arsen and {L}unts.
\newblock {\em Proc. Amer. Math. Soc.}, 138(4):1231--1242, 2010.

\bibitem{Srinivas}
V.~Srinivas.
\newblock The {H}odge characteristic.
\newblock {\em Preprint, 2002}.

\bibitem{Voisin}
C.~Voisin.
\newblock {\em Hodge theory and complex algebraic geometry. {I}}, volume~76 of
  {\em Cambridge Studies in Advanced Mathematics}.
\newblock Cambridge University Press, Cambridge, english edition, 2007.
\newblock Translated from the French by Leila Schneps.

\end{thebibliography}

\end{document}